\documentclass[11pt,a4paper,oneside]{amsart}

\newcounter{commentcounter}

\usepackage{amsmath} 
\usepackage{amssymb} 
\usepackage{amsthm} 
\usepackage{stmaryrd} 
\usepackage[english]{babel} 
\usepackage[font=small,justification=centering]{caption} 
\usepackage[nodayofweek]{datetime}
\usepackage[shortlabels]{enumitem} 
\usepackage[T1]{fontenc} 
\usepackage[utf8]{inputenc} 
\usepackage{ifthen} 
\usepackage{mathabx} 
\usepackage{mathtools} 
\usepackage[dvipsnames]{xcolor} 
\usepackage[pdftex,  colorlinks=true, pagebackref=true]{hyperref} 
    \hypersetup{urlcolor=RoyalBlue, linkcolor=RoyalBlue,  citecolor=black}
\usepackage{setspace} 
\usepackage{tikz-cd} 
\usepackage{xfrac} 
\usepackage[capitalize]{cleveref} 

\renewcommand*{\backref}[1]{}
\renewcommand*{\backrefalt}[4]
{
    \ifcase #1
        No citation in the text.
    \or
        Cited on Page #2.
    \else
        Cited on Pages #2.
    \fi
}

\makeatletter
\@namedef{subjclassname@1991}{Mathematical subject classification 1991}
\@namedef{subjclassname@2000}{Mathematical subject classification 2000}
\@namedef{subjclassname@2010}{Mathematical subject classification 2010}
\@namedef{subjclassname@2020}{Mathematical subject classification 2020}
\makeatother

\newtheorem{thm}{Theorem}[section]
\newtheorem{lemma}[thm]{Lemma}
\newtheorem{corollary}[thm]{Corollary}
\newtheorem{prop}[thm]{Proposition}

\newtheorem{thmx}{Theorem}

\theoremstyle{definition}

\newtheorem{remark}[thm]{Remark}

\newtheorem{quest}[thm]{Question}

\theoremstyle{plain}

    \newtheoremstyle{TheoremNum}
        {8.0pt plus 2.0pt minus 4.0pt}{8.0pt plus 2.0pt minus 4.0pt} 
        {\itshape} 
        {-0.15cm} 
        {\bfseries} 
        {.} 
        { }  
        {\thmname{#1}\thmnote{ \bfseries #3}}
    \theoremstyle{TheoremNum}

\newcommand*{\claimproofname}{My proof}

\DeclareMathOperator{\Aut}{\mathrm{Aut}}

\newcommand{\SL}{\mathrm{SL}}
\newcommand{\PSL}{\mathrm{PSL}}

\newcommand{\SU}{\mathrm{SU}}

\newcommand{\SO}{\mathrm{SO}}

\newcommand{\He}{\mathrm{He}}

\newcommand{\onto}{\twoheadrightarrow}

\def\Z{\mathbb{Z}}

\newcommand{\NN}{\mathbb{N}}
\newcommand{\ZZ}{\mathbb{Z}}
\newcommand{\CC}{\mathbb{C}}

\newcommand{\FF}{\mathbb{F}}

\usepackage{tikz}
\usetikzlibrary{arrows,quotes}
\tikzstyle{blackNode}=[fill=black, draw=black, shape=circle]

\DeclareMathOperator{\spin}{spin}
\newcommand{\cs}{\mathbin{\#}}
\newcommand{\spincs}{\mathfrak{s}}
\newcommand{\cptwo}{\CC P^2}
\newcommand{\interior}{\textup{int}}
\title[Simple groups and surface complements]{Simple groups and complements of smooth surfaces in simply connected $4$-manifolds}

\usepackage[foot]{amsaddr}
\usepackage{enumitem}
\author{Sam Hughes}
\address[S.~Hughes]{Mathematisches Institut, Rheinische Friedrich-Wilhelms-Universit\"at Bonn, Endenicher Allee 60, 53115 Bonn, Germany}
\email{sam.hughes.maths@gmail.com}

\author{Daniel Ruberman}
\address[D.~Ruberman]{Department of Mathematics, MS 050, Brandeis University, Waltham, MA 02454-9110, USA}
\email{ruberman@brandeis.edu}

\date{\today}

\begin{document}

\begin{abstract}
    For each integer $n$ we construct an oriented closed simply connected $4$-manifold $X$ admitting a smoothly embedded closed connected surface $\Sigma$ of self-intersection number $n$ such that the complement of the surface has non-trivial fundamental group.  This answers a question of Kronheimer in Kirby's 1997 problem list. The proof combines a topological construction with homological properties of simple groups such as Thompson's group $V$ and certain sporadic finite simple groups.
\end{abstract}

\maketitle

\section{Introduction}\label{sec:intro}
A basic invariant of an embedding of a surface $\Sigma$ in a $4$-manifold $X$ is the fundamental group of its complement, and it is reasonable to ask what groups can occur, especially if $X$ is simply connected. With no further restrictions on the embedding, there is a simple characterisation of such groups~\cite{kim-ruberman:complement}: A finitely presented group $G$ is $\pi_1(X-\Sigma)$ for a connected surface $\Sigma$ in some $X$ (not specified in advance) if and only if $H_1(G)$ is cyclic and $G$ is the normal closure of a single element; a question of Wiegold~\cite[Problem 5.52]{khukhro2023unsolved} asks if the second condition follows automatically from the first.

The construction in~\cite{kim-ruberman:complement} provides a surface (in fact a symplectically embedded surface in a symplectic $4$-manifold) with self-intersection $0$, but the problem is more challenging if one requires that the self-intersection is non-zero.  Indeed, a question of Kronheimer in Kirby's 1997 problem list~\cite[Problem 4.109]{kirby:problems96} asks, for simply connected $X$, whether $\pi_{1}(X-\Sigma)$ is trivial if $\Sigma$ has non-zero self-intersection that is square-free (all prime divisors appear once). The motivation is an observation of Kronheimer--Mrowka~\cite[Proposition 5.7]{kronheimer-mrowka:I} that in this situation, $\pi_{1}(X-\Sigma)$ would have no non-trivial representations in $\SO(3)$ or $\SU(2)$.  This fact suggested that possible fundamental groups of surface complements might be found among groups without unitary representations, and led to the consideration of Thompson's group $V$, discussed in \Cref{sec.V}.

In this note, we construct surfaces with non-simply connected complements and provide a negative answer to Kronheimer's question. The result applies to all non-zero self-intersections, not just square-free ones.

\begin{thmx}
    \label{thmx.main}
    For any $n\in\Z$, there is an oriented closed simply connected $4$-manifold $X$ and surface $\Sigma$ smoothly embedded in $X$ such that $\Sigma \cdot \Sigma = n$ and $\pi_1(X-\Sigma) \neq \{1\}$.   If $n$ is even we may take $X$ to be spin.
\end{thmx}

\begin{remark}\label{orient} We note here some topological preliminaries.
\begin{enumerate}[label=(\roman*)]
    \item All surfaces and 4-manifolds discussed in the paper will be assumed compact and connected. 
    \item The sign of the self-intersection $\Sigma\cdot\Sigma$  depends on a choice of orientation for $X$ which we make implicitly throughout.  Note that by reversing the orientation of $X$, the statement of Theorem~\ref{thmx.main} for $n$ is equivalent to the same statement for $-n$.
    \item The condition that the self-intersection of $\Sigma\subset X$ be square-free implies that the homology class $[\Sigma]\in H_2(X)$ is primitive. More generally, if $[\Sigma] =dA$ where $A$ is a primitive class in $H_2(X;\Z)$, then we call $d$ the multiplicity of $\Sigma$. As in ~\cite[Lemma 3.1]{hsiang-szczarba}, if $[\Sigma]$ has multiplicity $d$, then  $H_1(X-\Sigma)
    \cong \Z/d$ and $\Sigma \cdot \Sigma = d^2 A\cdot A$.
\end{enumerate}
    
\end{remark}

\subsection*{Acknowledgements}
This work has received funding from the European Research Council (ERC) under the European Union's Horizon 2020 research and innovation programme (Grant agreement No. 850930). The second author was partially supported by NSF grant DMS-1952790.  The first author thanks Ian Leary for helpful correspondence, and the second author thanks Peter Kronheimer for the reminder about his question.  Both authors thank the referees for the helpful comments which improved the paper.

\section{From group theory to surface complements}\label{sec.2}
Let $\Sigma$ be a closed orientable surface of genus at least one.  Choose an orientation for $\Sigma$, and 
let $N$ be the disk bundle over $\Sigma$ of Euler class $n$ with orientable total space. Its boundary, a circle bundle over $\Sigma$, will be denoted by $Y$. We note that for an embedding of $\Sigma$ with normal bundle $N$, the complement $X -\Sigma$ of $\Sigma$ deformation retracts onto the exterior of $\Sigma$, defined as $X - \interior(N)$; it is more convenient henceforth to work with the exterior.  The idea of the proof of \Cref{thmx.main} is to find a suitable manifold $W$ with $\partial W = Y$ to serve as the exterior of $\Sigma$ in $X$. In fact we will take $X$ to be $N \cup W$, and the goal is to choose $W$ so that $X$ is simply connected. 

To this end, we make use of the oriented bordism group $\Omega_k(BG)$, where $BG$ is the classifying space for the group $G$; see~\cite{atiyah:bordism-cobordism}, where $\Omega_k(\cdot)$ is denoted $\operatorname{MSO}_k(\cdot)$. The addition in this group is given by disjoint union. The elements of this group are equivalence classes of pairs $(M,f\colon M \to BG)$, where $M$ is a compact oriented $k$-manifold, and $(M,f)$ and $(M',f')$ are equivalent if they are cobordant via an oriented manifold $V$ and and a map $F\colon  V \to BG$ extending $f$ and $f'$. Note that if $M$ is a connected oriented closed $k$-manifold, a homomorphism $\varphi\colon \pi_1(M) \to G$ is induced by a map $M \to BG$, well-defined up to homotopy. The bordism class of this map will be denoted $[M,\varphi]\in \Omega_k(BG)
$; if $M$ is disconnected, then $\varphi$ would be a collection of homomorphisms of the fundamental groups of its components, and one simply adds the bordism elements. If there is an oriented cobordism $V$ between $M$ and $M'$ and a homomorphism from $\pi_1(V) \to G$ restricting to $\varphi$ and $\varphi'$ on $M$ and $M'$, respectively, then $[M,\varphi] = [M',\varphi']$.
 One defines, similarly, the spin bordism group $\Omega^{\spin}_k(BG)$ as spin cobordism classes of triples $(M,\spincs,f)$ where $\spincs$ is a spin structure on $M$ and $f\colon M \to BG$ as above.
 
The following lemma is a straightforward consequence of the Atiyah-Hirzebruch spectral sequences
\[E^2_{p,q}=H_p(X;\Omega_q(pt))\quad \text{and} \quad E^2_{p,q}=H_p(X;\Omega_q^{\spin}(pt))\] 
for the generalised homology theory of oriented bordism~\cite{atiyah-hirzebruch:AHSS,mccleary:SS} or spin bordism. The first part can be proved geometrically as in~\cite{gordon:G-signature}. Note~\cite[\S I.6]{Brown1982} that the group homology $H_*(G)$ is the same as the homology $H_*(BG)$ of the classifying space; we will use the former notation as it is more standard in the group theory literature.   

\begin{lemma}\label{bordism}
    The map $\Omega_3(BG) \to H_3(G;\Z)$ that assigns $\varphi_*([M])$ to a pair $[M,\varphi]$ is an isomorphism.  If $H_1(G;\Z) = H_2(G;\Z) = 0$, then the same is true for $\Omega_3^{\spin}(BG)$.
\end{lemma}

Fix an orientation of the fibres of $Y$.  Recall that the fundamental group of $Y$ is a central extension
\[\begin{tikzcd}
    \{1\} \arrow{r} &\Z \arrow{r} & \pi_1(Y) \arrow{r} &  \pi_{1}(\Sigma) \arrow{r} & \{1\}
\end{tikzcd}\]
and that the $\Z$ subgroup is the kernel of the map $\pi_1(Y) \to \pi_1(N)$. Denote by $\mu$ the generator of this $\Z$ subgroup. 

\begin{lemma} \label{create-complement}
    Suppose that there is a finitely presented perfect group $P$ and a homomorphism $\varphi: \pi_1(Y) \to P$ such that $\varphi(\mu)$ normally generates $P$ and the bordism class  $[Y,\varphi]\ \in \Omega_3(BP)$ is trivial.
Then there is an embedding of $N$ in a simply connected $4$-manifold $X$ such that the exterior $W =X - \interior(N)$ has fundamental group $P$. If $n$ is even, choose a  spin structure on $N$ with restriction $\sigma_Y$ to $Y$. If the bordism class of  $(Y,\sigma_Y,\varphi)$ in $\Omega^{\spin}_3(BP)$ is trivial, then $X$ may be chosen to be a spin manifold.
\end{lemma}
\begin{proof}
    By hypothesis, there is a compact $4$-manifold $W_0$ with boundary $Y$ and the portion connected by solid arrows in the diagram below.
\begin{center}
\begin{tikzcd}
    \pi_1(Y) \arrow{rr}{\varphi} \arrow{dr}[swap]{j_*} &   &P &&\\[4ex]
&  \pi_1(W_0) \arrow[dotted]{r}{i} \arrow{ur}{\Phi_0}& \pi_1(W_1)\arrow[dotted]{r}{q} \arrow[dotted,swap,two heads]{u}{\Phi_1}& \pi_1(W) \arrow[dotted,"\cong","\Phi"']{ul}
\end{tikzcd}
\end{center}
We modify $W_0$ in two stages to fill in the dotted portion of the diagram so that $\Phi$ is an isomorphism. Since $P$ is finitely presented, there is a closed $4$-manifold $W_P$ with $\pi_1(W_P)\cong P$. Replace $W_0$ with $W_0 \cs W_P$ and note that the homomorphism $\pi_1(W_0 \cs W_P) \cong \pi_1(W_0) * P \to P$ given by $\Phi_0$ on the first factor and the identity on the second factor is a surjection. Choose finitely many generators $\{x_j\}$ for $\pi_1(W_0)$.  Now do surgery on circles in $W_0 \cs W_P$ representing the elements $x_j^{-1}\Phi_0(x_j)$ to obtain a manifold $W_1$ with a surjection $\Phi_1$ as indicated. The homomorphism $i:\pi_1(W_0) \to \pi_1(W_1)$ is induced by the inclusion of the summand $W_0$ in $W_0 \cs W_P$.  As in~\cite[Theorem 1.4]{wall:surgery,wall:book2} do surgery on finitely many circles in $W_1$ to kill the kernel of $\Phi_1$, obtaining the manifold $W$ and isomorphism $\Phi$ as indicated.

The image of $\mu$ in $\pi_1(W)$ is taken to $\varphi(\mu)$, which is assumed to be non-trivial in $P$. By hypothesis, $\varphi(\mu)$ normally generates $\pi_1(W)$. On the other hand, $\mu$ is trivial in $\pi_1(N)$, so van Kampen's theorem says that $X = W \cup N$ is simply connected.

When $n$ is even, so that $N$ has a spin structure, we describe how the argument goes through to produce a simply connected spin manifold $X$. By hypothesis, the initial manifold $W_0$ has a spin structure extending the spin structure  on $Y$ that it acquires as the boundary of $N$. We need to see that the surgeries done in the preceding argument can be chosen so that at each stage, the new manifold inherits a spin structure. A more general statement can be found in the surgery-theory literature, for example in~\cite[Lemma 2]{kreck:surgery-duality}, but this might be challenging for the reader. So we provide a direct proof (valid in any dimension greater than $2$) of exactly the statement we need.

A normal framing of a circle  $\gamma$  embedded in an $n$-manifold $Z^n$ is equivalent (up to isotopy) to an embedding $\varphi\colon S^1 \times D^{n-1} \to Z$ such that $\varphi(S^1 \times 0) = \gamma$. Write $Z_\varphi$ for the result of surgery using this framing. Then there is a cobordism  $C_\varphi$ from $Z$ to $Z_\varphi$ obtained by adding a $2$-handle $D^2\times D^{n-1}$ to $Z \times I$ along $\varphi(S^1 \times D^{n-1}) \times \{1\}$. 
\\[1ex]
{\bf Claim:} Let $Z^n$ be a spin manifold with $n \geq 3$ and let $\gamma$ be a circle embedded in $Z$. Then there is a framing $\varphi$ of $\gamma$ so that $Z_\varphi$ has a spin structure extending the spin structure on the exterior of $\gamma$.\\[1ex]
{\bf Proof of Claim:} Recall that the set of spin structures on a manifold $M$ has a free and transitive action of the group $H^1(M;\Z/2)$. In particular, $S^1 \times D^{n-1}$ has two spin structures, one of which (the `bounding spin structure') is the restriction of the unique spin structure on $D^2 \times D^{n-1}$, and the other of which does not extend over $D^2 \times D^{n-1}$. 

Starting with one framing $\varphi$, a second framing is obtained by composing $\varphi$ with the self-diffeomorphism $\tau$ of $S^1 \times D^{n-1}$ built from the non-trivial element of $\pi_1(\SO(n-1))$. It is straightforward to show that the pull-back of the bounding spin structure on $S^1 \times D^{n-1}$ is the non-bounding one, and vice versa. Now consider a spin structure $\sigma_Z$ on $Z$ that we'd like to extend over the cobordism $C_\varphi$. If the pull-back of $\sigma_Z$ to $S^1 \times D^{n-1}$ is the bounding one, then $\sigma_Z$ extends over the cobordism $C_\varphi$, yielding the desired spin structure on $Z_\varphi$. If, on the other hand,  $\varphi^*\sigma_Z$ is the non-bounding spin structure, the spin structure $(\varphi\circ \tau)^*\sigma_Z$ is the bounding one, yielding a spin structure on $Z_{\varphi\circ\tau}$. 

This establishes the claim, and completes the proof of the lemma.
\end{proof}

As remarked at the top of this section, the embedding of $N$ in $X$ yields an embedding of $\Sigma$ for which the complements have the same fundamental group. 
Hence, to complete the proof of Theorem~\ref{thmx.main}, it suffices to find groups and homomorphisms satisfying the hypotheses of Lemma~\ref{create-complement}; this is carried out in the next two sections of the paper. The first of these makes use of Thompson's infinite simple group, while the second uses an assortment of finite simple groups.

\section{From circle bundles to Thompson's group}\label{sec.V}
A group $G$ is \emph{type $\mathsf{F}_n$} if it admits a $K(G,1)$ with finite $n$-skeleton and \emph{type $\mathsf{F}_\infty$} if it is type $\mathsf{F}_n$ for all $n\geq 0$.  A group $G$ is \emph{type $\mathsf{FP}_n$} if it admits a projective resolution $P_\bullet \to \Z$ over $\Z G$ with $P_i$ finitely generated for $i\leq n$ and \emph{type $\mathsf{FL}_n$} if the $P_i$ can be taken to be finitely generated free modules for $i\leq n$.  We say $G$ is \emph{type $\mathsf{FP}_\infty$} it is type $\mathsf{FP}_n$ for all $n$ and \emph{type $\mathsf{FL}_\infty$} it is type $\mathsf{FL}_n$ for all $n$.  Note that $\mathsf{FP}_n$ and $\mathsf{FL}_n$ are equivalent \cite[VIII, Prop 4.1]{Brown1982} and that Wall proved in \cite{Wall1965} that a finitely presented group of type $\mathsf{FL}_n$ is type $\mathsf{F}_n$.

\begin{prop}\label{prop.V}
    Let $S^1\to Y\to \Sigma_g$ exhibit $Y$ as a circle bundle over a surface of genus $g\geq 1$.  Let $H=\pi_1 (Y)$ and let $\mu\in H$ be a generator of $\pi_1(S^1)< H$.  Then, there exists a group $V$ and a homomorphism $\psi\colon H\to V$ such that 
    \begin{enumerate}
        \item $V$ is finitely presented, in fact $\mathsf{F}_\infty$;
        \item $H_i(V;\Z)=0$ for $i=1,2,3$; in particular $\Omega_3(BV)=\Omega^{\spin}_3(BV)=0$;
        \item $\psi(\mu)$ normally generates $V$.
    \end{enumerate}
\end{prop}
\begin{proof}
    We let $V$ denote Thompson's group $V$ (see \cite{Higman1974} for a definition).  We recount the following facts:
    \begin{enumerate}
        \item $V$ is a simple group \cite{Higman1974};
        \item $V$ is acyclic \cite{SzymikWahl2019}, so $H_i(V;\Z)=0$ for all $i\geq 1$;
        \item $V$ contains every finite group as a subgroup \cite{Higman1974};\label{Vfinite}
        \item $V$ is finitely presented \cite{Higman1974};
        \item $V$ is type $\mathsf{FP}_\infty$ \cite{Brown1987};
        \item $V$ is type $\mathsf{F}_\infty$, which follows from the previous two items.
    \end{enumerate}

By \cite{Hempel1987} the group $H$ is residually finite.  Thus, we can find a finite quotient $h\colon H\onto Q$ such that $h(\mu)$ is non-trivial.  Since every finite group is a subgroup of $V$ we may embed $Q$ into $V$ via some homomorphism $i\colon Q\to V$.  We define $\psi$ to be $i\circ h$.  Now, as $V$ is simple, every non-trivial element of $V$ normally generates it. Thus, $\psi(\mu)$ normally generates $V$.
\end{proof}

\begin{remark}
When the Euler number of the circle bundle $n$ is greater than one in absolute value, we do not have to appeal to the residual finiteness of $H$. For in this case, $H$ has a surjection onto $\Z/n$ taking $\mu$ to a generator. By item \eqref{Vfinite} above, $\Z/n$ is a subgroup of $V$, which provides the desired $\psi$. It is slightly more delicate to find explicit homomorphisms when $n = \pm 1$; in the next section we do this with $V$ replaced by various finite simple groups.
\end{remark}

\begin{proof}[Proof of \Cref{thmx.main}]
     Combining \Cref{prop.V} with \Cref{create-complement} completes the proof of \Cref{thmx.main}.  Here the complements $X-\Sigma$ have fundamental group $V$.
\end{proof}

\begin{remark}[Higman--Thompson groups]
    In fact we can build more examples out of the Higman--Thompson groups $V_{m,r}$.  In \cite{SzymikWahl2019}, the authors show $H_\ast(V_{m,r};\Z)\cong H_\ast(\Omega^\infty_0\mathbf{M}_{m-1};\Z)$, where the second object is the homology of the zeroth component of the infinite loop space of the mod $m-1$ Moore spectrum.  The relevance for us is Propositions 6.1 and 6.2 of \emph{ibid}; there it is shown that when $m$ is even and $p$ is the smallest prime dividing $m-1$ we have $\widetilde H_d(V_{m,r};\Z)=0$ for $d<2p-3$ and $H_{2p-3}(V_{m,r};\Z)=\Z/p$.  In particular, these groups can be realised as the fundamental group of a surface complement with one caveat: if $3|(m-1)$, then one must map to a finite group $Q$ with $|H_3(Q;\Z)|$ coprime to $3$ to ensure the bordism class vanishes.
\end{remark}

\begin{remark}\label{rem linear reps}
    Higman--Thompson groups have no non-trivial linear representations.  Indeed, by Mal'cev's Theorem \cite{Maltsev1965}, a finitely generated linear group is residually finite.  Hence, any finitely generated group admitting a non-trivial linear representation also admits a non-trivial finite quotient.  But $V_n$ is simple and so does not admit any non-trivial finite quotients.
\end{remark}

\section{Finite simple groups}\label{S:simple}
In this section, to demonstrate the flexibility of the method, we construct examples of surface complements with fundamental group isomorphic to certain finite simple groups.  This demonstrates that the the fundamental group of a surface complement need not be trivial or infinite.  It is also noteworthy because finite groups have many linear representations, as opposed to the Higman--Thompson groups, which have no non-trivial linear representations (see \Cref{rem linear reps}). Note that the finite groups we consider have no non-trivial representations in $\SO(3)$ or $\SU(2)$, so we do not contradict \cite[Proposition~5.7]{kronheimer-mrowka:I}.

\subsection*{Self-intersection number one}
We first consider the case where the disk bundle over the surface $\Sigma$ has Euler number equal to one or minus one.  

This section will use a number of notions from finite group theory so before we begin we establish some notation.  The group $\He_3(p)$ is the ($3\times3)$-\emph{Heisenberg group} over $\FF_p$ and is isomorphic to the group of ($3\times3)$-uni-triangular matrices over the finite field $\FF_p$ of $p$-elements.  Here a uni-triangular matrix is upper triangular with ones on the diagonal.  We will refer to a number of sporadic simple groups: the \emph{Mathieu groups} $M_{22}$ and $M_{23}$, the \emph{Higman--Sims} group $HS$, the \emph{Held group} $He$, the \emph{McLaughlin group} $McL$, the \emph{Janko groups} $J_3$ and $J_4$, the \emph{Lyons group} $Ly$, the \emph{Conway groups} $Co_3,$ $Co_2$, and $Co_1$, the \emph{Suzuki group} $Suz$, the \emph{O'Nan group} {\it O'N}, the \emph{Fisher group} $Fi_{22}$ and the \emph{Monster group} $M$.  Our notation for these groups is fairly standard and follows the Atlas of finite simple groups \cite{ATLASfsg}.

\begin{prop}\label{prop sporadic}
    Let $\Sigma$ be a closed oriented surface of genus at least one.  Each sporadic simple group in \Cref{table} can be realised as the fundamental group of a surface complement $X-\Sigma$ with $\Sigma\cdot \Sigma =\pm 1$, where $X$ is some closed oriented simply connected $4$-manifold.
\end{prop}

\begin{table}[h]
\begin{tabular}{|l|llllllll|}
\hline
            & $M_{22}$ & $M_{23}$ & $HS$ & $He$ & $McL$ & $J_3$ & $J_4$ & $Ly$   \\
            \hline
$H_2(P;\Z)$ & 12       & 0        & 2    & 0    & 0     & 3     & 0     & 0      \\
$H_3(P;\Z)$ & 0        & 0        & 2+2  & 12   & 0     & 15    & 0     & 0      \\
$\He_3(p)$  & 2        & 2        & 5    & 7    & 2,3,5 & 2     & 2,3   & 2,3,5  \\
\hline \hline
            & $Co_3$  &$Co_2$  & $Co_1$    & $Suz$   & {\it O'N}    & $Fi_{22}$    & $M$   &       \\
            \hline
$H_2(P;\Z)$ & 0      & 0         & 2         & 6       & 3        & 6          & 0 &            \\
$H_3(P;\Z)$ & 6      & 4         & 12        & 4       & 8        & 1          & \multicolumn{2}{l|}{$24+(\leq 4)$}  \\
$\He_3(p)$  & 5      & 3,5       & 5         & 3       & 7        & 2,3        & 5, 7    &    \\
\hline
\end{tabular}
\caption{Some sporadic simple groups $P$ where the third homology is known (except in the case of the Monster group $M$ where there are three possibilities) and which admit a $\He_3(p)$ subgroup with $p$ coprime to the order of the third homology of $P$.  The notation $m+n$ in the homology rows is interpreted as $\Z/m \oplus \Z/n$; a $\leqslant k$ means a subgroup of $\Z/k$ appears a summand. The $p$ for which such a subgroup exists are listed in the $\He_3(p)$ row.}
\label{table}
\end{table}

\begin{proof}
Let $H$ denote the fundamental group of the associated circle bundle over $\Sigma$.  Since $\Sigma\cdot\Sigma=\pm 1$,  $H$ has a presentation
\[\left\langle a_1,b_1,\dots,a_g,b_g,z\ |\ \prod_{i=1}^g[a_i,b_i]=z,\ [a_j,z]=[b_j,z]=1 \text{ for }1\leq j\leq g\right\rangle. \]
The group $H$ admits a surjection onto $\He_3(\Z)$, the $3$-dimensional integer Heisenberg group (uni-triangular matrices), with presentation
\[\left\langle a,b,z\ |\ [a,b]=z,\ [a,z]=[b,z]=1 \right\rangle.\]
The homomorphism is given by $a_1\mapsto a$, $b_1\mapsto b$, $a_i,b_i\mapsto 1$ for $i\geq2$, and of course $z\mapsto z$.  For each prime $p$, the group $\He_3(\Z)$ admits a surjection onto the $p$-group $\He_3(p)$ by considering the modulo $p$ reduction of matrices in $\He_3(\Z)$.  Note that when $p=2$ this group is the dihedral group on $4$ points $D_4$ containing $8$ elements.

We want to apply \Cref{create-complement}; the following conditions are sufficient on some finite simple group $P$:
\begin{enumerate}
    \item $\He_3(p)\leqslant P$
    \item $H_3(P;\Z)$ has no $p$-torsion.
\end{enumerate}
Indeed, we want the class $(Y,\varphi)$ to be trivial in $\Omega_3(BP)$, but this class is in the image of the composition
\[H_3(Y;\Z)\to H_3(\pi_1(Y);\Z) \to H_3(\He_3(p);\Z) \to H_3(P;\Z)\]
and the group $H_3(\He_3(p);\Z)$ is annihilated by multiplication by $p^3$.  So the last map is trivial and hence $[Y,\varphi]=0\in\Omega_3(BP)$.  Thus, to complete the proof we just have to confirm the veracity of the information in \Cref{table}.

In \Cref{table}, the existence or non-existence of the $\He_3(p)$ subgroups is easily verified by consulting the ATLAS of finite simple groups \cite{ATLASfsg}.  The homology computations were obtained from the following sources: the low-dimensional homology of the Mathieu groups can be found in \cite{DutourEllis2009} and the remaining computations are contained in \cite{Johnson-Freyd--Truemann2019}.  Note that for the Monster group $M$, it is known that $H_3(M;\Z)$ is a subgroup of $\Z/24 \oplus \Z/4$ but the exact group is not known.  The groups $M_{11}$, $M_{12}$, $M_{24}$, $J_1$, $J_2$ have been omitted from \Cref{table} due to having no $\He_3(p)$ subgroup with $p$ coprime to $|H_3(P;\Z)|$.  The groups $Ru$, $HN$, $Th$, $Fi_{23}$, $Fi_{24}'$, and $B$ have been omitted from \Cref{table} because neither $H_3(G;\Z)$ nor its prime divisors are known.
\end{proof}

\subsection*{Self-intersection number greater than one}
As remarked in \Cref{sec.V}, when the self-intersection number $n$ satisfies $|n|>1$ there is a surjection $\pi_1(Y)\onto \Z/n$ such that the image of $\mu$ is a generator of $\Z/n$.  Here we will show that for $n$ divisible by a prime $p\geq 7$, we can realise the simple group $\PSL_2(p)$ as a surface complement.

Recall that $\SL_2(p)$ is the group of invertible $2\times 2$-matrices over $\FF_p$ with determinant equal to $1$ and that $Z(\SL_2(p))=\{\pm I_2\}$.  We identify $\PSL_2(p)$ with $\SL_2(p)/Z(\SL_2(p))$.

\begin{prop}\label{prop psl}
    For each prime $p\geq 7$, integer $g\geq 0$, and $n\in \Z$ such that $p|n$, the group $\PSL_2(p)$ can be realised as the fundamental group of a surface complement $X-\Sigma$, where $X$ is a closed oriented simply connected $4$-manifold, $\Sigma$ has genus $g$, and $\Sigma\cdot \Sigma =\pm n$.
\end{prop}

\begin{proof}
The key point is that $\Z/n$ surjects onto $\Z/p$ and this latter group is isomorphic to a Sylow $p$-subgroup of the simple group $P=\PSL_2(p)$.  Before we prove the proposition we show that $H_3(P;\Z)$ has no elements of order $p$.

Note that a Sylow $p$-subgroup $Q$ of $P$ is isomorphic to $\Z/p$. We may take $Q$ to be the image of the upper triangular matrices with $1$s on the diagonal under the projection $\SL_2(p)\to \PSL_2(p)$.  The action of $N_P(Q)$ on $Q$ factors through a homomorphism $\phi: N_P(Q) \to \Aut(Q)$ whose kernel is $C_P(Q)$.  It is straightforward to compute that $|N_P(Q)|=\frac{1}{2}p(p-1)$ and $C_P(Q) = Q$, so the image of $\phi$ is the unique subgroup of index 2 in the cyclic group $\Aut(Q)$.

By \cite[Lemma~1]{Swan1960} (see also \cite[Theorems~6.6 and 6.8]{AdemMilgram1994}), the $p$-part of $H^k(P;\Z)$ for $k\geq 1$ is the fixed points of the $N_P(Q)$-action on $H^k(\Z/p;\Z)$.   We have that $H_1(Q;\Z)\cong Q$ and $H_2(Q;\Z)$ is trivial, so by the Universal Coefficient Theorem, $H^2(Q;\Z)\cong Q$.  Thus the action of $\Aut(Q) \cong (\Z/p)^\times$ on $H^2(Q;\Z)\cong \Z/p$ is given by multiplication in $\Z/p$.   Now, by naturality of the cup product, we see that the action on $H^{2k}(Q;\Z)$ is the $k$th power of the action on $H^2(Q;\Z)$.  Hence, the degrees $\ell=2k$ with $\frac{1}{2}(p-1)|k$, where $H^{2k}(P;\Z)$ contains $p$-torsion, are exactly the degrees for which the action is trivial. In particular, for $p\geq 7$, by applying the Universal Coefficient Theorem, we see that the group $H_3(P;\Z)$ has no elements of order $p$.

Let $Y$ be an oriented circle bundle over a closed genus $g\geq0$ surface of Euler number $n$.  Consider the composite $\varphi\colon \pi_1(Y)\twoheadrightarrow \Z/n \twoheadrightarrow \Z/p \rightarrowtail P$ which is non-trivial on a generator of the centre of $\pi_1(Y)$.  We claim the bordism class $[Y,\varphi]$ vanishes in $\Omega_3(BP)\cong H_3(P;\Z)$.  Indeed, since $H_3(\Z/p;\Z)\cong\Z/p$ and $H_3(P;\Z)$ has no elements of order $p$ we see that map $H_3(\Z/p;\Z) \to H_3(P;\Z)$ is zero.  The proposition follows from applying \Cref{create-complement}.
\end{proof}

\begin{corollary}
    For any $n\in\Z$, there is an oriented closed simply connected $4$-manifold $X$ and surface $\Sigma$ smoothly embedded in $X$ such that $\Sigma \cdot \Sigma = n$ and $\pi_1(X-\Sigma)$ is a non-abelian finite simple group.
\end{corollary}
\begin{proof}
    After applying \Cref{prop sporadic} and \Cref{prop psl} we are left with the case that $n$ is divisible only by $2$, $3$, or $5$.  In these cases we use homomorphisms to the Lyons group $Ly$ given, respectively, by $\pi_1 (Y)\onto \Z/2\rightarrowtail  Ly$, $\pi_1(Y)\onto \Z/3\rightarrowtail Ly$, or $\pi_1 (Y)\onto \Z/5\rightarrowtail Ly$ which are all seen to vanish on third homology since $H_3(Ly;\Z)=0$; see \Cref{table}.
\end{proof}

\subsection*{A classification?}
Whilst we have provided many examples of finite simple groups which can appear as complements of surfaces in simply connected $4$-manifolds, we fall a long way short of a complete classification. 
To this end we raise the following question.

\begin{quest}
    Let $G$ be a finite simple group.  Which $(g,n)\in\NN_{\geq1}\times\ZZ$ ensure there exists a simply connected $4$-manifold and $\Sigma_g\subset X$, a smoothly embedded surface of genus $g$, with $\Sigma_g\cdot\Sigma_g = n$ such that $\pi_1(X-\Sigma)\cong G$?
\end{quest}

\section{The non-primitive case}\label{multiple}
    The arguments and examples in the preceding sections produce surface complements $W = X- \Sigma$ with $H_1(W) = 0$ but for which $\pi_1(W)$ is non-trivial, answering the question of Kronheimer. For non-primitive classes (to which his question does not apply) the corresponding statement would be that there are surface complements for which the fundamental group of the complement is not the cyclic group $\Z/d$, where $d$ is the multiplicity (defined in Remark~\ref{orient}) of the homology class $[\Sigma] \in H_2(X)$. In this short section we show that this holds by presenting examples of surfaces whose complements have non-abelian fundamental groups. The construction is explicit and does not require any of our group-theoretic arguments.  
\begin{prop}
    For any non-zero integers $m$ and $d$, there is a simply connected $4$-manifold containing a smoothly embedded surface of multiplicity $d$ and self-intersection $d^2m$ whose complement has non-abelian fundamental group. 
\end{prop}    
    \begin{proof}
    Recall Zeeman's $d$-twist spinning construction~\cite{zeeman:twist}, which from a knot $K$ in $S^3$ produces a fibred knot $\tau_d(K)$ in $S^4$. The fibre is $\Sigma_d(K)$, the $d$-fold branched cyclic cover of $S^3$ branched along $K$, minus a ball, and the monodromy is a generator of the covering transformations. 

  Write $n = d^2m$, and assume without loss of generality (as in Remark~\ref{orient}) that $m > 0$. The first step is to find a surface $\Sigma_0$ embedded in a simply connected manifold $X$ with simply connected complement, with $\Sigma_0 \cdot \Sigma_0 = m$. For instance, one could take an algebraic curve in $\cptwo$ of sufficiently high degree so that its self-intersection is greater than $m$, and then blow up enough points to lower the self-intersection to $m$.  Let $X$ be the resulting blow-up of $\cptwo$. Note that that the blown-up curve now meets a $2$-sphere (any of the exceptional curves in the blow-up) in a single point and hence has simply connected complement.  Then $d$ times the homology class of $\Sigma_0$ is represented by a smoothly embedded surface $\Sigma_1 \subset X$ with $\pi_1(X - \Sigma_1) \cong \Z/d$.  

    Now replace $\Sigma_1$ by $\Sigma = \Sigma_1 \# \tau_d(K)$ where $K$ is any non-trivial knot in $S^3$; we will show that the fundamental group of the complement of $\Sigma$ is non-abelian.    It is argued in~\cite{kim:surfaces} that the fundamental group of $X - \Sigma$ is $\pi_1(S^3- K)/\mu^d$, and that this contains the fundamental group of the $d$-fold cyclic cover of $S^3$ branched along $K$ as an index $d$ subgroup.   By the solution of the generalized Smith conjecture~\cite[page 4]{bass-morgan:smith}, the branched cover of $S^3$ along a non-trivial knot cannot have trivial fundamental group. (This also follows from the orbifold theorem~\cite{cooper-hodgson-kerckhoff:orbifolds,boileau-porti:cyclic,boileau-leeb-porti:orbifolds}.)  In other words, $\pi_1(S^3- K)/\mu^d$ has a non-trivial index $d$ subgroup. If $\pi_1(S^3- K)/\mu^d$ were abelian, it would have to be isomorphic to its abelianization, $\Z/d$, and hence the only index $d$ subgroup it contains would be trivial. It follows that $\pi_1(S^3- K)/\mu^d$, and hence the fundamental group of $X - \Sigma$, is non-abelian.
    \end{proof}

\bibliographystyle{halpha}
\bibliography{refs.bib}

@book {ATLASfsg,
    AUTHOR = {Conway, J. H. and Curtis, R. T. and Norton, S. P. and Parker,
              R. A. and Wilson, R. A.},
     TITLE = {{$\mathbb{ATLAS}$} of finite groups},
      NOTE = {Maximal subgroups and ordinary characters for simple groups,
              With computational assistance from J. G. Thackray},
 PUBLISHER = {Oxford University Press, Eynsham},
      YEAR = {1985},
     PAGES = {xxxiv+252},
      ISBN = {0-19-853199-0},
   MRCLASS = {20D05 (20-02)},
  MRNUMBER = {827219},
MRREVIEWER = {R.\ L.\ Griess},
}

@book {AdemMilgram1994,
    AUTHOR = {Adem, Alejandro and Milgram, R. James},
     TITLE = {Cohomology of finite groups},
    SERIES = {Grundlehren der mathematischen Wissenschaften [Fundamental
              Principles of Mathematical Sciences]},
    VOLUME = {309},
 PUBLISHER = {Springer-Verlag, Berlin},
      YEAR = {1994},
     PAGES = {viii+327},
      ISBN = {3-540-57025-X},
   MRCLASS = {20J06 (20-01 55P20)},
  MRNUMBER = {1317096},
MRREVIEWER = {David\ Benson},
       DOI = {10.1007/978-3-662-06282-1},
       URL = {https://doi.org/10.1007/978-3-662-06282-1},
}

@article {atiyah:bordism-cobordism,
    AUTHOR = {Atiyah, M. F.},
     TITLE = {Bordism and cobordism},
   JOURNAL = {Proc. Cambridge Philos. Soc.},
  FJOURNAL = {Proceedings of the Cambridge Philosophical Society},
    VOLUME = {57},
      YEAR = {1961},
     PAGES = {200--208},
      ISSN = {0008-1981},
   MRCLASS = {57.10},
  MRNUMBER = {126856},
MRREVIEWER = {A.\ Dold},
       DOI = {10.1017/s0305004100035064},
       URL = {https://doi.org/10.1017/s0305004100035064},
}

@incollection {atiyah-hirzebruch:AHSS,
    AUTHOR = {Atiyah, M. F. and Hirzebruch, F.},
     TITLE = {Vector bundles and homogeneous spaces},
 BOOKTITLE = {Proc. {S}ympos. {P}ure {M}ath., {V}ol. {III}},
     PAGES = {7--38},
 PUBLISHER = {Amer. Math. Soc., Providence, RI},
      YEAR = {1961},
   MRCLASS = {57.30 (14.52)},
  MRNUMBER = {139181},
MRREVIEWER = {R.\ Bott},
}

@book {Brown1982,
    AUTHOR = {Brown, Kenneth S.},
     TITLE = {Cohomology of groups},
    SERIES = {Graduate Texts in Mathematics},
    VOLUME = {87},
      NOTE = {Corrected reprint of the 1982 original},
 PUBLISHER = {Springer-Verlag, New York},
      YEAR = {1994},
     PAGES = {x+306},
      ISBN = {0-387-90688-6},
   MRCLASS = {20J05 (20-02)},
  MRNUMBER = {1324339},
}

@inproceedings {Brown1987,
    AUTHOR = {Brown, Kenneth S.},
     TITLE = {Finiteness properties of groups},
 BOOKTITLE = {Proceedings of the {N}orthwestern conference on cohomology of
              groups ({E}vanston, {I}ll., 1985)},
   JOURNAL = {J. Pure Appl. Algebra},
  FJOURNAL = {Journal of Pure and Applied Algebra},
    VOLUME = {44 (1-3)},
      YEAR = {1987},
     PAGES = {45--75},
      ISSN = {0022-4049,1873-1376},
   MRCLASS = {20J05 (11F75 20F05 22E40)},
  MRNUMBER = {885095},
MRREVIEWER = {Ralph\ Strebel},
       DOI = {10.1016/0022-4049(87)90015-6},
       URL = {https://doi.org/10.1016/0022-4049(87)90015-6},
}

@article {DutourEllis2009,
    AUTHOR = {Dutour Sikiri\'{c}, Mathieu and Ellis, Graham},
     TITLE = {Wythoff polytopes and low-dimensional homology of {M}athieu
              groups},
   JOURNAL = {J. Algebra},
  FJOURNAL = {Journal of Algebra},
    VOLUME = {322},
      YEAR = {2009},
    NUMBER = {11},
     PAGES = {4143--4150},
      ISSN = {0021-8693,1090-266X},
   MRCLASS = {20J06},
  MRNUMBER = {2556144},
MRREVIEWER = {V.\ D.\ Mazurov},
       DOI = {10.1016/j.jalgebra.2009.09.031},
       URL = {https://doi.org/10.1016/j.jalgebra.2009.09.031},
}

@incollection {gordon:G-signature,
    AUTHOR = {Gordon, C. McA.},
     TITLE = {On the {$G$}-signature theorem in dimension four},
 BOOKTITLE = {\`{A} la recherche de la topologie perdue},
    SERIES = {Progr. Math.},
    VOLUME = {62},
     PAGES = {159--180},
 PUBLISHER = {Birkh\"auser Boston, Boston, MA},
      YEAR = {1986},
   MRCLASS = {58G10 (57S17)},
  MRNUMBER = {900251},
}

@incollection {Hempel1987,
    AUTHOR = {Hempel, John},
     TITLE = {Residual finiteness for {$3$}-manifolds},
 BOOKTITLE = {Combinatorial group theory and topology ({A}lta, {U}tah,
              1984)},
    SERIES = {Ann. of Math. Stud.},
    VOLUME = {111},
     PAGES = {379--396},
 PUBLISHER = {Princeton Univ. Press, Princeton, NJ},
      YEAR = {1987},
      ISBN = {0-691-08409-2; 0-691-08410-6},
   MRCLASS = {57M05 (20E26 20F34 57N10)},
  MRNUMBER = {895623},
}

@book {Higman1974,
    AUTHOR = {Higman, Graham},
     TITLE = {Finitely presented infinite simple groups},
    SERIES = {Notes on Pure Mathematics},
    VOLUME = {No. 8},
 PUBLISHER = {Australian National University, Department of Pure
              Mathematics, Department of Mathematics, I.A.S., Canberra},
      YEAR = {1974},
     PAGES = {vii+82},
   MRCLASS = {20F05 (20B25)},
  MRNUMBER = {376874},
MRREVIEWER = {F.\ Levin},
}

@incollection {hsiang-szczarba,
    AUTHOR = {Hsiang, W. C. and Szczarba, R. H.},
     TITLE = {On embedding surfaces in four-manifolds},
 BOOKTITLE = {Algebraic topology ({P}roc. {S}ympos. {P}ure {M}ath., {V}ol.
              {XXII}, {U}niv. {W}isconsin, {M}adison, {W}is., 1970)},
    SERIES = {Proc. Sympos. Pure Math.},
    VOLUME = {Vol. XXII},
     PAGES = {97--103},
 PUBLISHER = {Amer. Math. Soc., Providence, RI},
      YEAR = {1971},
   MRCLASS = {57D95},
  MRNUMBER = {339239},
MRREVIEWER = {J.\ M.\ Boardman},
}

@article {Johnson-Freyd--Truemann2019,
    AUTHOR = {Johnson-Freyd, Theo and Treumann, David},
     TITLE = {Third homology of some sporadic finite groups},
   JOURNAL = {SIGMA Symmetry Integrability Geom. Methods Appl.},
  FJOURNAL = {SIGMA. Symmetry, Integrability and Geometry. Methods and
              Applications},
    VOLUME = {15},
      YEAR = {2019},
     PAGES = {Paper No. 059, 38},
      ISSN = {1815-0659},
   MRCLASS = {20D08 (20J06)},
  MRNUMBER = {3990846},
MRREVIEWER = {V.\ D.\ Mazurov},
       DOI = {10.3842/SIGMA.2019.059},
       URL = {https://doi.org/10.3842/SIGMA.2019.059},
}

@article {kim-ruberman:complement,
    AUTHOR = {Kim, Hee Jung and Ruberman, Daniel},
     TITLE = {Smooth surfaces with non-simply-connected complements},
   JOURNAL = {Algebr. Geom. Topol.},
  FJOURNAL = {Algebraic \& Geometric Topology},
    VOLUME = {8},
      YEAR = {2008},
    NUMBER = {4},
     PAGES = {2263--2287},
      ISSN = {1472-2747},
   MRCLASS = {57R57 (57N13)},
  MRNUMBER = {MR2465741},
       DOI = {10.2140/agt.2008.8.2263},
       URL = {http://dx.doi.org/10.2140/agt.2008.8.2263},
}

@incollection{kirby:problems96,
author = "Kirby, R.C.",
title = "Problems in Low--Dimensional Topology",
booktitle = "Geometric Topology",
editor = "Kazez, W.",
publisher = "American Math.\ Soc./International Press",
address = "Providence",
year = 1997}

@article {kreck:surgery-duality,
    AUTHOR = {Kreck, Matthias},
     TITLE = {Surgery and duality},
   JOURNAL = {Ann. of Math. (2)},
  FJOURNAL = {Annals of Mathematics. Second Series},
    VOLUME = {149},
      YEAR = {1999},
    NUMBER = {3},
     PAGES = {707--754},
      ISSN = {0003-486X,1939-8980},
   MRCLASS = {57R67 (57R65)},
  MRNUMBER = {1709301},
MRREVIEWER = {Laurence\ R.\ Taylor},
       DOI = {10.2307/121071},
       URL = {https://doi.org/10.2307/121071},
}

@article{kronheimer-mrowka:I,
    AUTHOR = {Kronheimer, P. B. and Mrowka, T. S.},
     TITLE = {Gauge theory for embedded surfaces. {I}},
   JOURNAL = {Topology},
  FJOURNAL = {Topology. An International Journal of Mathematics},
    VOLUME = {32},
      YEAR = {1993},
    NUMBER = {4},
     PAGES = {773--826},
      ISSN = {0040-9383},
   MRCLASS = {57R57 (57N13 57R40 57R55 58D29)},
  MRNUMBER = {1241873},
MRREVIEWER = {Ronald\ J.\ Stern},
       DOI = {10.1016/0040-9383(93)90051-V},
       URL = {https://doi.org/10.1016/0040-9383(93)90051-V},
}

@Article{Maltsev1965,
 Author = {Mal'tsev, A. I.},
 Title = {On the faithful representation of infinite groups by matrices},
 FJournal = {Translations. Series 2. American Mathematical Society},
 Journal = {Transl., Ser. 2, Am. Math. Soc.},
 ISSN = {0065-9290},
 Volume = {45},
 Pages = {1--18},
 Year = {1965},
 Language = {English},
 DOI = {10.1090/trans2/045/01},
 zbMATH = {3254031},
 Zbl = {0158.02905}
}

@article {Swan1960,
    AUTHOR = {Swan, Richard G.},
     TITLE = {The {$p$}-period of a finite group},
   JOURNAL = {Illinois J. Math.},
  FJOURNAL = {Illinois Journal of Mathematics},
    VOLUME = {4},
      YEAR = {1960},
     PAGES = {341--346},
      ISSN = {0019-2082},
   MRCLASS = {20.25 (18.00)},
  MRNUMBER = {122856},
MRREVIEWER = {S.-T.\ Hu},
       URL = {http://projecteuclid.org/euclid.ijm/1255456050},
}

@article {SzymikWahl2019,
    AUTHOR = {Szymik, Markus and Wahl, Nathalie},
     TITLE = {The homology of the {H}igman-{T}hompson groups},
   JOURNAL = {Invent. Math.},
  FJOURNAL = {Inventiones Mathematicae},
    VOLUME = {216},
      YEAR = {2019},
    NUMBER = {2},
     PAGES = {445--518},
      ISSN = {0020-9910,1432-1297},
   MRCLASS = {19D23 (20J05)},
  MRNUMBER = {3953508},
MRREVIEWER = {Fernando\ Muro},
       DOI = {10.1007/s00222-018-00848-z},
       URL = {https://doi.org/10.1007/s00222-018-00848-z},
}

@book {mccleary:SS,
    AUTHOR = {McCleary, John},
     TITLE = {A user's guide to spectral sequences},
    SERIES = {Cambridge Studies in Advanced Mathematics},
    VOLUME = {58},
   EDITION = {Second},
 PUBLISHER = {Cambridge University Press, Cambridge},
      YEAR = {2001},
     PAGES = {xvi+561},
      ISBN = {0-521-56759-9},
   MRCLASS = {55Txx (18G40)},
  MRNUMBER = {1793722},
MRREVIEWER = {Frank\ Neumann},
}

@book {wall:book2,
    AUTHOR = {Wall, C. T. C.},
     TITLE = {Surgery on compact manifolds},
    SERIES = {Mathematical Surveys and Monographs},
    VOLUME = {69},
   EDITION = {Second},
      NOTE = {Edited and with a foreword by A. A. Ranicki},
 PUBLISHER = {American Mathematical Society, Providence, RI},
      YEAR = {1999},
     PAGES = {xvi+302},
      ISBN = {0-8218-0942-3},
   MRCLASS = {57R67 (19J25 57-02)},
  MRNUMBER = {1687388},
       DOI = {10.1090/surv/069},
       URL = {https://doi.org/10.1090/surv/069},
}

@article {wall:surgery,
    AUTHOR = {Wall, C. T. C.},
     TITLE = {Surgery of non-simply-connected manifolds},
   JOURNAL = {Ann. of Math. (2)},
  FJOURNAL = {Annals of Mathematics. Second Series},
    VOLUME = {84},
      YEAR = {1966},
     PAGES = {217--276},
      ISSN = {0003-486X},
   MRCLASS = {57.20 (57.10)},
  MRNUMBER = {212827},
MRREVIEWER = {A.\ Haefliger},
       DOI = {10.2307/1970519},
       URL = {https://doi.org/10.2307/1970519},
}

@article {zeeman:twist,
    AUTHOR = {Zeeman, E. C.},
     TITLE = {Twisting spun knots},
   JOURNAL = {Trans. Amer. Math. Soc.},
  FJOURNAL = {Transactions of the American Mathematical Society},
    VOLUME = {115},
      YEAR = {1965},
     PAGES = {471--495},
      ISSN = {0002-9947,1088-6850},
   MRCLASS = {55.20},
  MRNUMBER = {195085},
MRREVIEWER = {L.\ Neuwirth},
       DOI = {10.2307/1994281},
       URL = {https://doi.org/10.2307/1994281},
}

@article {kim:surfaces,
    AUTHOR = {Kim, Hee Jung},
     TITLE = {Modifying surfaces in 4-manifolds by twist spinning},
   JOURNAL = {Geom. Topol.},
  FJOURNAL = {Geometry and Topology},
    VOLUME = {10},
      YEAR = {2006},
     PAGES = {27--56},
      ISSN = {1465-3060,1364-0380},
   MRCLASS = {57R57 (14J80 57R40 57R65 57R95)},
  MRNUMBER = {2207789},
MRREVIEWER = {Hongzhu\ Gao},
       DOI = {10.2140/gt.2006.10.27},
       URL = {https://doi.org/10.2140/gt.2006.10.27},
}

@book {bass-morgan:smith,
     TITLE = {The {S}mith conjecture},
    SERIES = {Pure and Applied Mathematics},
    VOLUME = {112},
    EDITOR = {Morgan, John W. and Bass, Hyman},
      NOTE = {Papers presented at the symposium held at Columbia University,
              New York, 1979},
 PUBLISHER = {Academic Press, Inc., Orlando, FL},
      YEAR = {1984},
     PAGES = {xv+243},
      ISBN = {0-12-506980-4},
   MRCLASS = {57-06 (57N10 57S17 57S25)},
  MRNUMBER = {758459},
MRREVIEWER = {Allan\ Edmonds},
}

@misc{khukhro2023unsolved,
      title={Unsolved Problems in Group Theory. The {K}ourovka {N}otebook}, 
      author={E. I. Khukhro and V. D. Mazurov},
      year={2023},
      eprint={1401.0300},
      archivePrefix={arXiv},
      primaryClass={math.GR}
}

@article{boileau-leeb-porti:orbifolds,
    AUTHOR = {Boileau, Michel and Leeb, Bernhard and Porti, Joan},
     TITLE = {Geometrization of 3-dimensional orbifolds},
   JOURNAL = {Ann. of Math. (2)},
  FJOURNAL = {Annals of Mathematics. Second Series},
    VOLUME = {162},
      YEAR = {2005},
    NUMBER = {1},
     PAGES = {195--290},
      ISSN = {0003-486X},
     CODEN = {ANMAAH},
   MRCLASS = {57M50 (53C23 57N10)},
  MRNUMBER = {2178962 (2007f:57028)},
MRREVIEWER = {Darryl McCullough},
       DOI = {10.4007/annals.2005.162.195},
       URL = {http://dx.doi.org/10.4007/annals.2005.162.195},
}

@article{boileau-porti:cyclic,
    AUTHOR = {Boileau, Michel and Porti, Joan},
     TITLE = {Geometrization of 3-orbifolds of cyclic type},
      NOTE = {Appendix A by Michael Heusener and Porti},
   JOURNAL = {Ast\'erisque},
  FJOURNAL = {Ast\'erisque},
    volume  = {No. 272},
      YEAR = {2001},
     PAGES = {208 pages},
      ISSN = {0303-1179},
   MRCLASS = {57M50 (53C23 57M60 57N10)},
  MRNUMBER = {1844891 (2002f:57034)},
MRREVIEWER = {Kevin P. Scannell},
}

@book {cooper-hodgson-kerckhoff:orbifolds,
    AUTHOR = {Cooper, Daryl and Hodgson, Craig D. and Kerckhoff, Steven P.},
     TITLE = {Three-dimensional orbifolds and cone-manifolds},
    SERIES = {MSJ Memoirs},
    VOLUME = {5},
      NOTE = {With a postface by Sadayoshi Kojima},
 PUBLISHER = {Mathematical Society of Japan, Tokyo},
      YEAR = {2000},
     PAGES = {x+170},
      ISBN = {4-931469-05-1},
   MRCLASS = {57M50 (57N10)},
  MRNUMBER = {1778789 (2002c:57027)},
MRREVIEWER = {Danny C. Calegari},
}

@article {Wall1965,
    AUTHOR = {Wall, C. T. C.},
     TITLE = {Finiteness conditions for {${\rm CW}$}-complexes},
   JOURNAL = {Ann. of Math. (2)},
  FJOURNAL = {Annals of Mathematics. Second Series},
    VOLUME = {81},
      YEAR = {1965},
     PAGES = {56--69},
      ISSN = {0003-486X},
   MRCLASS = {55.40},
  MRNUMBER = {171284},
MRREVIEWER = {A.\ Dold},
       DOI = {10.2307/1970382},
       URL = {https://doi.org/10.2307/1970382},
}

\end{document}